\documentclass[12pt,reqno]{amsart}
\usepackage{a4wide}
\usepackage[T1]{fontenc}
\usepackage{amssymb,amsmath,amsthm,mathtools}
\usepackage{mathrsfs} 
\usepackage{enumitem}
\usepackage[dvipsnames]{xcolor}
\usepackage{graphicx}
\usepackage{hyperref, color}
\usepackage{cleveref} 
\hypersetup{colorlinks=true, linkcolor=blue, citecolor=teal, urlcolor=cyan}
\usepackage{comment}
\usepackage{amsrefs}
\usepackage{tikz}
\usetikzlibrary{calc,angles,quotes,patterns}

\newtheorem{theorem}{Theorem}[section]
\newtheorem{lemma}[theorem]{Lemma}
\newtheorem{corollary}[theorem]{Corollary}

\newtheorem{fact}[theorem]{Fact}

\newtheorem*{claim*}{Claim}
\newtheorem{problem}[theorem]{Problem}
\newcounter{maintheorem}

\newtheorem{mainth}[maintheorem]{Theorem}
\theoremstyle{remark}

\theoremstyle{definition}
\newtheorem{definition}[theorem]{Definition}

\numberwithin{equation}{section}
\makeatother

\newcommand{\R}{\mathbb{R}}

\newcommand{\N}{\mathbb{N}}

\newcommand{\nn}[1]{{\left\vert\kern-0.25ex\left\vert\kern-0.25ex\left\vert #1 
		\right\vert\kern-0.25ex\right\vert\kern-0.25ex\right\vert}}

\renewcommand{\leq}{\leqslant}
\renewcommand{\geq}{\geqslant}
\renewcommand{\tilde}{\widetilde}
\renewcommand{\epsilon}{\varepsilon}

\DeclareMathOperator{\dist}{dist}

\renewcommand{\H}{\mathcal{H}}

\newcommand{\C}{\mathfrak{c}}

\newcommand{\inte}{\mathrm{int}\,}

\def\IPA#1#2#3#4{\Pi^{#1,#2}_ {#4}\left(#3\right)}

\usepackage[colorinlistoftodos,textsize=tiny]{todonotes}

\usepackage{bclogo}

\title[Regularity and $d$-stability]{Characterization of  regularity via variational \smallskip \\
	stability of alternating projection sequences}

\author[F.~Battistoni]{F. Battistoni}
\address[F.~Battistoni]{Dipartimento di Matematica per le Scienze economiche, finanziarie ed attuariali, Universit\`a Cattolica del Sacro Cuore, 20123 Milano, Italy \newline
	\href{https://orcid.org/0000-0003-2119-1881}{\texttt{ORCID:0000-0003-2119-1881}}}
\email{francesco.battistoni@unicatt.it}

\author[A.~Daniilidis]{A. Daniilidis}
\address[A.~Daniilidis]{Institute of Statistics and Mathematical Methods in Economics, E105-04, TU Wien, Wiedner Hauptstra{\ss }e 8, A-1040 Wien, Austria \newline
	\href{https://orcid.org/0000-0003-4837-694X}
	{\texttt{ORCID:0000-0003-4837-694X}}}
\email{aris.daniilidis@tuwien.ac.at}

\author[C.~A.~De~Bernardi]{C. A. De Bernardi}
\address[C.~A.~De~Bernardi]{Dipartimento di Matematica per le Scienze economiche, finanziarie ed attuariali, Universit\`a Cattolica del Sacro Cuore, 20123 Milano, Italy \newline
	\href{https://orcid.org/0000-0002-9654-1324}{\texttt{ORCID:0000-0002-9654-1324}}}
\email{carloalberto.debernardi@unicatt.it, carloalberto.debernardi@gmail.com}

\author[E.~Miglierina]{E. Miglierina}
\address[E.~Miglierina]{Dipartimento di Matematica per le Scienze economiche, finanziarie ed attuariali, Universit\`a Cattolica del Sacro Cuore, 20123 Milano, Italy \newline
	\href{https://orcid.org/0000-0003-3493-8198}{\texttt{ORCID:0000-0003-3493-8198}}}
\email{enrico.miglierina@unicatt.it}

\begin{document}
	
	\begin{abstract} The notion of regular pair $(A,B)$ for two nonempty closed convex subsets $A$ and~$B$ of a Hilbert space $\H$ was introduced by Borwein and Bauschke in 1993 to ensure convergence (in norm) of the alternating projection method to some point of the best approximation set. In 2022, De Bernardi and Miglierina showed that regularity of the pair $(A,B)$ guarantees, additionally, the convergence for any variational perturbation of the alternating projection method, provided the corresponding best approximation sets are bounded. In this work, we show that the converse assertion is also true. Moreover, this converse assertion holds without requiring the best approximation sets to be bounded.
	\end{abstract}
	\maketitle
	
	regularity, Attouch-Wets convergence.
	
	\vspace{0.5cm}
	
	\medskip
	
	\tableofcontents

	\section{Introduction}
	
	\noindent Let $A$ and $B$ be two nonempty closed convex sets in a Hilbert space $\H$. The 2-set convex feasibility problem consists of finding a pair of points $\bar{a}\in A$ and $\bar{b}\in B$ realizing the distance between $A$ and $B$, that is, 
    $$\|\bar{a}-\bar{b}\|=\dist(A,B):=\inf_{a\in A}\inf_{b\in B}\|a-b\|.$$
    Denoting by $\dist(a,B):=\inf_{b\in B}\|a-b\|$, a necessary and sufficient condition for the existence of such a pair is that the {\em best approximation sets} 
	\begin{equation}\label{eq:E,F}
		E:=\{a\in A:\, \dist(a,B)=\dist(A,B)\},
		\qquad F:=\{b\in B:\, \dist(b,A)=\dist(A,B)\}
	\end{equation}
	are nonempty. In this work, we shall assume that this condition is always satisfied and consequently the problem is well-posed. Notice that this happens whenever the intersection of the sets $A$ and $B$ is nonempty (in this case, $E=F=A\cap B$).
	\smallskip\newline
	The method of alternating projections is the simplest iterative procedure for finding a solution of the convex feasibility problem and it goes back to von Neumann \cite{vonNeumann50}: let us denote by $P_A$ and $P_B$ the projections onto the sets $A$ and $B$, respectively. Then for any starting point $a_0\in \H$, consider the {\em alternating projection sequences} $\{a_n\}_{n\ge 1} \subset A$ and $\{b_n\}_{n\geq 1} \subset B$ defined inductively by 
	\begin{equation}\label{eq:alternProjDef}
		b_n=P_{B}(a_{n-1})\ \ \text{and}\ \ a_n=P_{A}(b_n)\ \ \ \ \ (n\in\N).    
	\end{equation}
	If the sequences $\{a_n\}$ and $\{b_n\}$ converge in norm, then we say that the method of alterna\-ting projections (also known as the von Neumann method) converges, that is, there exist $\bar a \in E$, $\bar b \in F$ such that $\lim_{n\to \infty} a_n = \bar a$ and $\lim_{n\to\infty} b_n = \bar b$. This is always the case in finite dimensions provided the sets $E$, $F$ in~\eqref{eq:E,F} are nonempty, while in infinite dimensions, the sequences $\{a_n\}$ and $\{b_n\}$ are only weakly converging, even if $A\cap B \neq \emptyset$, see ~\cite{BauschkeBorwein94}*{Theorem~4.8}. Indeed, in the separable Hilbert space $\ell_2$, a celebrated example of Hundal in 2004 (\cite{Hundal}) shows that in general the alternating projections method may fail to converge in norm. This example was later simplified by Kopeck\'a \cite{Kop08} and motivated further investigations on conditions that ensure the convergence of the von Neumann method. \smallskip\newline
	Besides the case in which $A$ and $B$ are linear subspaces (this was the original setting studied by von Neumann himself, where obviously one has $0\in A\cap B =E=F$), one of the 
	best known and most important sufficient conditions for the convergence of such a method is the notion of {\em regularity} for the pair $(A,B)$ (see forthcoming Definition~\ref{def: regularity}), introduced by Bauschke and Borwein in \cite{BauschkeBorwein93}. This notion (which already requires the best approximation sets $E$ and $F$ to be nonempty) guarantees the norm convergence of the sequences defined in~\eqref{eq:alternProjDef}.
	\smallskip\newline
	In the more recent papers \cites{DebeMiglStab,DebeMiglMol,DebeMigl}, the authors studied stability properties of the alternating projection method. More precisely, for any two sequences of nonempty closed convex sets
	$\{A_n\}_{n\geq 1}$ and $\{B_n\}_{n\geq 1}$ such that $\lim_{n\to \infty} A_n = A$ and $\lim_{n\to\infty}B_n= B$ for the Attouch-Wets {variational} convergence (see forthcoming Definition~\ref{def:AW}) and any initial point $a_0\in \H$, they considered the {\em perturbed alternating projection sequences}  $\{a_n\}_{n\ge 1}$ and~$\{b_n\}_{n \ge 1}$, defined as follows:
	\begin{equation} \label{eq:var-alt-proj}
		b_n=P_{B_n}(a_{n-1})\ \ \ \text{and}\ \ \  a_n=P_{A_n}(b_n) \ \ \ \ \ \ \ \ \ (n\in\N).
	\end{equation}
	In~\cite{DebeMiglStab} it was shown that if $E,F$ are nonempty and bounded, then the aforementioned regularity condition on the pair $(A,B)$ not only guarantees the norm convergence of the alternating projection sequence in~\eqref{eq:alternProjDef}, but also the norm convergence of its {\em variational} version given by~\eqref{eq:var-alt-proj}. To be more precise, using the notation of~\eqref{eq:E,F} let us recall the notion of $d$-stability introduced in \cite{DebeMiglStab}. 
	\begin{definition}[$d$-stability]\label{def:perturbedseq} We say that the pair $(A,B)$ of two nonempty closed convex subsets of a Hilbert space $\H$ is {\em $d$-stable}, whenever all perturbed alternating projection sequences $\{a_n\}_{n\geq 1}$ and~$\{b_n\}_{n\geq 1}$ given in~\eqref{eq:var-alt-proj} satisfy $$\lim_{n\to\infty}\mathrm{dist}(a_n,E) = 0 \qquad \text{and} \qquad \lim_{n\to\infty}\mathrm{dist}(b_n,F)= 0.$$ 
	\end{definition}
	\noindent The main result of \cite{DebeMiglStab} asserts that if the pair $(A,B)$ is regular and the sets $E$ and $F$ are (nonempty and) bounded, then the pair $(A,B)$ is $d$-stable. The assumption on the boundedness of the best approximation sets $E$ and $F$ is necessary for this statement, since Example~5.2 in \cite{DebeMiglStab} provides a  regular pair $(A,B)$ with $E=F=A\cap B\neq \emptyset$ being unbounded, which is not $d$-stable. 	
	 In the same paper, the authors asked whether the inverse implication holds, stating the following open problem:
	\begin{problem}\label{open problem} 
		Is every $d$-stable pair $(A,B)$ necessarily regular?	
	\end{problem}
	
	\noindent The main aim of the present paper is to show that Problem~\ref{open problem} has a positive answer, and this answer does not require boundedness of the best approximation sets. 
	In particular, we establish the following result.
	\begin{mainth}\label{main theorem} Let $A,B$ be nonempty closed convex subsets of $\H$. If  the pair $(A,B)$  is $d$-stable, then the pair $(A,B)$ is regular. 
	\end{mainth}
	
	\begin{corollary}\label{main corollary} Let $A,B$ be nonempty closed convex subsets of $\H$ and assume that the best approximation sets $E$ and $F$ are  bounded. 
    Then the following are equivalent:
\begin{enumerate}
    \item the pair $(A,B)$ is regular;
    \item the pair $(A,B)$ is $d$-stable. 
\end{enumerate}
		\end{corollary}
	\noindent Theorem \ref{main theorem} establishes implication (ii)$\Longrightarrow$(i) (see forthcoming Theorems~\ref{th: separated case} and~\ref{th: general case}). The other implication was previously established in~\cite{DebeMiglStab}*{Theorem~4.9}. Altogether, Theorem~A, Corollary~\ref{main corollary} and Example~5.2 in  \cite{DebeMiglStab}  provide a complete characterization of regularity by means of variational stability of alternating projection sequences. \smallskip\newline
	Regularity of the pair $(A,B)$ is in fact stronger than mere convergence of alternating projection sequences, see \cite[Example~5.5]{BauschkeBorwein93}. Indeed, assuming that the best approximation sets are bounded, regularity turns out to be equivalent to a good behaviour of the alternating projections with respect to perturbations of the two sets $A$ and $B$, suitable for a study involving perturbations of the original sets, for which it could be easier to compute the projections. In specific cases, regular convex feasibility problems could also lead to the development of algorithmic procedures, designed for the study of convex sets $A$ and $B$ whose description is affected by some errors due to uncertainty in the data. On the other hand, our result suggests avoiding this variational approach in the case of absence of regularity (see Section~\ref{sec: problems} for more details).	\smallskip\newline
	We resume this introduction with a brief description of the structure of the paper. In Section~\ref{SEction notations} we fix our notation and review preliminary notions. Definitions and properties of the Hausdorff and, respectively, the Attouch-Wets set convergence are therein recalled, together with relevant properties of metric projections and of alternating projections method in a Hilbert space. In particular, the relationship between regularity and $d$-stability obtained in \cite{DebeMiglStab} is recalled. In Section~\ref{sec: main}, we prove the main result of the paper, namely, the fact that $d$-stability implies regularity for a given convex feasibility problem. After two key technical lemmas (namely, Lemma~\ref{lemma: geometrical fact} and Lemma~\ref{lemma:  black box}), the proof of our main result will be divided into two cases: first, we consider the case where $A\cap B$ is the empty set, then we move to the case where $A\cap B$ is nonempty, where the proof becomes more involved. Finally, in Section~\ref{sec: problems} we present some additional remarks, conclusions and open questions.

	
	\section{Notation and preliminaries }\label{SEction notations}
	Throughout this paper, we will always work with a Hilbert space~$\H$ equipped with an inner product $\langle\cdot,\cdot\rangle$ and its induced norm $\|\cdot\|$. We denote by $\mathbb{B}_{\H}$ and $\mathbb{S}_{\H}$ the closed unit ball and the unit sphere of $\H$, respectively. We denote by $\H^*$ the dual space of $\H$, even if the space $\H^*$ can be represented by $\H$ itself. This notation allows us to emphasize the role of linear functionals in several arguments, making them clearer. Further, we denote by $\mathrm{Fix}(T)$ the set of fixed points of an operator $T:\H\rightarrow \H$. \smallskip\newline
    If $\alpha>0$, $x\in X$, and $A,B\subset \H$, we set as usual
	$$x+A:=\{x+a;\, a\in A\},\quad \alpha A:=\{\alpha a;\, a\in A\},\quad A+B:=\{a+b;\, a\in A,\, b\in B\}.$$ 
	For any two points $x,y\in \H$, we denote by $[x,y]$ the closed segment in $\H$ with
	endpoints $x$ and $y$ and we set $(x,y)=[x,y]\setminus\{x,y\}$ for the
	corresponding ``open'' segment. The segment $(x,y]$ is defined similarly. For a subset $A$ of $\H$, we denote by $\inte(A)$ the interior of $A$. 
	We further define the distance of a point $x\in X$ to a set $A$ as follows:
	$$\dist(x,A) :=\inf_{a\in A} \|a-x\|.$$ 
	We are going to use the following  simplified notation. If $f$ is a real-valued function defined on $\H$, we denote
	$$
	[f\leq\alpha]:=\{x\in \H:\,  f(x)\leq\alpha\}\quad\text{for $\alpha\in\R$.}
	$$
	The sets $[f\geq\alpha]$ and $[f=\alpha]$ are defined in a similar way. Finally, we recall that $$\ker f:=\{x\in \H:f(x)=0\},$$ whenever $f$ is a linear functional defined on $\H$.

	\subsection{Hausdorff and Attouch-Wets convergence of sequences of sets} Let us now review two notions of convergence for a sequence of sets (for a more detailed overview of this topic, see, e.g.,~\cite{Beer}). By 
    $\C(\H)$ we denote the family of all nonempty closed subsets of~$\H$.
	Let us introduce the (extended) Hausdorff metric $\mathrm{D}_{\mathrm{H}}$ in
	$\C(\H)$. For $A,B\in\C(\H)$, we define the excess of $A$ over $B$
	as follows:
	$$e(A,B): = \sup_{a\in A} \mathrm{dist}(a,B).$$
	\noindent Notice that the above definition still makes sense if one of the two sets is empty: indeed, if $A\neq\emptyset$ and $B=\emptyset$ we set $e(A,B)=\infty$, while if $A=\emptyset$ we set $e(A,B)=0$. Furthermore, we define the Hausdorff distance 	between sets $A$ and $B$ as follows:
	$$\mathrm{D}_{\mathrm{H}}(A,B):=\max \bigl\{ \, e(A,B),\, e(B,A) \,\bigr\}.$$
	The above metric gives rise to our first notion of set convergence. 
	\begin{definition}[Hausdorff convergence] We say that a sequence $\{A_n\}$ in $\C(\H)$ 
		converges \textit{in the Hausdorff sense} to a set $A\in\C(\H)$ if $$\textstyle \lim_{n\to\infty} \mathrm{D}_{\mathrm{H}}(A_n,A) = 0\,.$$
	\end{definition}
	\noindent According to \cite{LUCC}*{Theorem~8.2.12} the sequence $\{A_n\}_{n \geq 1}$ in $\C(\H)$ is Hausdorff converging to a set $A \in \C(\H)$ if and only if 
	$$
	\sup_{x\in X}\,|\dist(x,A_n)-\dist(x,A)| \,\underset{n\to\infty}{\longrightarrow} 0\,. 
	$$
	We now consider a weaker notion of convergence, the so-called Attouch-Wets convergence (see,
	e.g., \cite{LUCC}*{Definition~8.2.13}). To this end, fix $r>0$, 
	$A,B\in\C(\H)$ and set:
	\begin{eqnarray*}
		e_r(A,B) &:=& e(A\cap r\,\mathbb{B}_{\H}, \,B) \,\,\in[0,\infty),\\
		\mathrm{D}_{\mathrm{H},r}(A,B) &:=& \max\,\{e_r(A,B),\, e_r(B,A)\}.
	\end{eqnarray*}
	The above family of pseudo-distances gives rise to the following notion of convergence:
	\begin{definition}[Attouch-Wets convergence]\label{def:AW} We say that a sequence $\{A_n\}_{n\geq 1}$ in $\C(\H)$ converges in the \textit{Attouch-Wets sense} to a set $A\in\C(\H)$ if for every $r>0$ we have:
		$$\textstyle \lim_{n\to \infty} \mathrm{D}_{\mathrm{H},\,r}(A_n,A)= 0\,.$$
	\end{definition}
	\noindent The Attouch-Wets convergence can be seen as a "localization" of the Hausdorff convergence. Indeed, according to \cite{LUCC}*{Theorem~8.2.14}, a sequence $\{A_n\}_{n\geq 1}$ in $\C(\H)$ is Attouch-Wets converging to a set $A\in \C(\H)$ if and only if for every $r>0$
	$$
	\sup_{x\in r\mathbb{B}_{\H}}\,|\dist(x,A_n)-\dist(x,A)|\,\underset{n\to\infty}{\longrightarrow} 0\,.
	$$
	The following property relates Hausdorff and Attouch-Wets convergences.
	\begin{fact}\label{fact: AW-intersection balls}
		Let $A\in\C(\H)$, $\{A_n\}_{n\geq 1}\subset \C(\H)$, $\{r_n\}_n \subset (0,+\infty)$ increasingly converging to~$+\infty$, and $\{\delta_n\}_n \subset (0,+\infty)$ converging to $0$. Assume that $\mathrm{D}_{\mathrm{H}}(A\cap r_n\mathbb{B}_{\H}, A_n)\leq \delta_n$, whenever $n\geq 1$. Then, for every $r>0$, the inequality $\mathrm{D}_{\mathrm{H},r}(A,A_n)\leq \delta_n$ eventually holds as $n\to\infty$. In particular,    the sequence $\{A_n\}_{n\geq 1}$ is Attouch-Wets converging to $A$.  
	\end{fact}
	\begin{proof}
		Let $r\in\N$. For every $n\in \N$, we clearly have
		$$
		e_r(A_n,A)=e(A_n\cap r\mathbb{B}_{\H},A)\leq e(A_n, A) \leq e(A_n,A\cap r_n\mathbb{B}_{\H})\leq \delta_n.
		$$
		Moreover, if $n_0\in \N$ is such that $r_{n_{0}}>r$, then for all $n\geq n_0$ we have 
		$$ 
		e_r(A,A_n)=e(A\cap r\mathbb{B}_{\H},A_n)\leq e(A\cap r_n\mathbb{B}_{\H},A_n)\leq \delta_n\,.
		$$
		We have shown that $\mathrm{D}_{\mathrm{H},r}(A,A_n)\leq \delta_n$ eventually holds as $n\to\infty$. The final conclusion, about the Attouch-Wets convergence of $\{A_n\}_{n\geq 1}$, immediately follows.
	\end{proof}
	
	\subsection{Projections in Hilbert spaces and the alternating projections method}
	The notions of distance between two convex sets and of projection of a point onto a convex set of a Hilbert space play a fundamental role in our paper. The projection onto a nonempty closed convex subset~$C$ maps any point $x_0\in \H$ to its nearest point in $C$, denoted by $P_C (x_0)$. 
	We recall the following result, usually named {\em variational characterization of the projection onto}  $C$. Let $c_0\in C$ and $x_0\in \H$. Then $c_0=P_C(x_0)$ if and only if 
	\begin{equation*}\label{eq:variat_charact}
		\langle x_0-c_0,c-c_0\rangle\leq 0,\quad \text{for all }\  c\in C.
	\end{equation*}
	We recall that the projection $P_C$ is a nonexpansive map from $\H$ to $C$, i.e., it holds \\$\|P_C(x)-P_C(y)\|\leq \|x-y\|$ (see, e.g., \cite{LUCC}*{Proposition~10.4.8}). 
	
	In the sequel, we shall need the following elementary fact, the proof of which is left to the reader.
	\begin{fact}\label{fact: carabinieri per proiezioni} Let $C, D\subset\H$ be  nonempty closed convex sets such that $C\subset D$. Then:
		\begin{enumerate}
			\item[$\mathrm{(a)}$] If $C$ is an affine set then $P_C$ is an affine function.
			\item[$\mathrm{(b)}$] If  $b,p\in\H$ are such that   $p=P_{D}(b)$ and $ p\in C$, then  $p=P_{C}(b)$.
		\end{enumerate}
		
	\end{fact}
	\smallskip

	Given $n\in\N$ and two convex sets $A,B\in \C(\H)$, let us consider the nonexpansive map  $\Pi^{{A},{B}}_n:\H\to\H$, defined by
	$$\IPA{{A}}{{B}}{x}{0}=x,\quad\IPA{{A}}{{B}}{x}{n}=\underbrace{P_{{A}} P_{B}\ldots P_{{A}} P_{B}}_{n \text{ times}}(x),\qquad x\in\H.$$
	Given $A$ and $B$ in $\C(\H)$, let $\{a_n\}_n$, $\{b_n\}_n$ be the corresponding alternating projection sequences as in~\eqref{eq:alternProjDef}, starting from the initial point $x\in\mathcal{H}$, then we can rewrite them in terms of the above operator $\IPA{{A}}{{B}}{\,\cdot\,}{n}$ as follows:
	$$
	a_n = \IPA{{A}}{{B}}{x}{n}, \qquad b_{n+1}= P_B(\IPA{{A}}{{B}}{x}{n}).
	$$
	
	Let us now consider two nonempty closed convex subsets $A$ and $B$ and define the {\em displacement vector}
	\begin{equation} \label{eq:displ}
		v:=P_{\overline{B-A}}(0) \qquad\text{for the pair } \, (A,B)\,.
	\end{equation}
Recalling~\eqref{eq:E,F} we have that the (possibly empty) best approximation sets $E,F$ are disjoint if and only if $A\cap B=\emptyset$. Notice that if $E$ and $F$ are nonempty, then there exist $\bar a \in E$, $\bar b \in F$ such that $v=\bar b - \bar a$. In case $A\cap B\neq\emptyset$, then $E=F=A\cap B$ and the displacement vector for the pair $(A,B)$ is null. Furthermore:
	\begin{fact}[{\cite{BauschkeBorwein93}*{Fact~1.1}}]\label{fact: BB93} Suppose that $\H$ is a Hilbert space and that $A,B$ are nonempty closed convex subsets of $\H$, such that the corresponding best approximation sets $E$ and $F$ are nonempty. Then:
		\begin{enumerate}
			\item $\|v\|=\mathrm{dist}(A,B)$ and $E+v=F$; 
			\item $E=\mathrm{Fix}(P_A P_B)=A\cap(B-v)$ and $F=\mathrm{Fix}(P_B P_A)=B\cap(A+v)$;
			\item For every $\bar a \in E$ and $\bar b \in F$ we have: 
			$$P_B \bar a=P_F \bar a =\bar a +v \qquad \text{ and } \qquad P_A \bar b =P_E \bar b =\bar b-v.$$
		\end{enumerate}
	\end{fact} 
	
	\subsection{Regularity and $d$-stability for a pair of convex sets} \label{Section regularity}
	We shall now state the notion of regularity for a pair of nonempty closed convex sets $A$ and $B$, which has been introduced in~\cite{BauschkeBorwein93} to ensure convergence in norm for the alternating projection algorithm (see also~\cite{BorweinZhu}). 
    
	\begin{definition}[regular pair]\label{def: regularity} Let $A$ and $B$ be two nonempty closed convex subsets of a Hilbert space $\H$, such that the corresponding best approximation sets $E$ and $F$ are nonempty.
		The pair $(A,B)$ is called {\em regular} if for every $\epsilon>0$ there exists $\delta>0$ such that 
			for every $x\in \H$ we have:
			\begin{equation}\label{eq:BB}
				\max\left\{\,\mathrm{dist}(x,A),\mathrm{dist}(x,B-v)\,\right\}\,\leq\,\delta \quad\implies\quad \mathrm{dist}(x,E)\leq \epsilon\,.
			\end{equation}			
	\end{definition}
	We refer the reader to~\cite{BauschkeBorwein93} for concrete examples of pair of sets satisfying (or failing to satisfy) regularity property. \smallskip\newline

	The next result follows from \cite{BauschkeBorwein93}*{Theorem~3.7} and \cite{DebeMiglStab}*{Theorem~4.9}. Indeed, recalling Definition~\ref{def:perturbedseq} we have: 
	\begin{theorem}
		Let $A,B$ be nonempty closed convex subsets of a Hilbert space $\H$ so that the pair $(A,B)$ is regular. The following assertions hold:
		\begin{enumerate}
			\item  the alternating projection method converges in norm;
			\item if $E$, $F$ are bounded, the pair $(A,B)$ is $d$-stable.
		\end{enumerate}
	\end{theorem}


	\section{Main Results}\label{sec: main}
	In this section, our aim is to prove Theorem~A. To this end, we shall need some intermediate technical lemmas, which we are going to obtain progressively. Let us start with the following important intermediate result that provides a framework under which a sequence generated by the alternating projection method remains in a 2-dimensional space. 
\begin{lemma} [keeping iterations in a 2-dim space] \label{lemma: geometrical fact} Let $\gamma\geq0$, let $z,w$ be two nonzero elements of a Hilbert space $\H$ such that $\langle z, w\rangle=0$ and define:
		\begin{equation}\textstyle 
			u=\frac{z}{\|z\|}, \quad \alpha :=\frac{\|z\|}{\|w\|}, \quad x\mapsto f(x)=\langle u,x\rangle\quad, \quad x\mapsto \widehat{f}(x)=\langle u+\alpha\frac{w}{\|w\|},x\rangle .
		\end{equation}
		Let further $\widehat{A},\widehat{B}$ be nonempty closed convex subsets of $\H$ such that
		\begin{equation}
			[z,w]\, \subseteq \widehat A \, \subseteq \left[\widehat{f}\geq\widehat{f}(w)\right]\qquad\text{and}\qquad [-\gamma u,w-\gamma u]\subseteq \widehat B\subseteq[f\leq -\gamma].
		\end{equation}
		Then starting from $a_0=0$ the alternating projection method generates sequences
$$a_n=\IPA{\widehat A}{\widehat B}{0}{n},\quad b_n=P_{\widehat B}(a_{n-1}),\quad n\geq 1, \quad  $$
such that if $\gamma=0$ we have:
$$\underset{n\to\infty}{\lim} a_n = \underset{n\to\infty}{\lim} b_n = w \in \widehat{A}\cap \widehat{B},$$ 
while if $\gamma\neq 0$, there exists $m\in\N\cup\{0\}$
satisfying $$\|a_m-w\|\leq \,\gamma\,\alpha\, \frac{\|w-z\|}{\|w\|}.$$
	\end{lemma}

	\begin{figure}[!ht]
		\begin{center}
			\begin{tikzpicture}
			contenu    
			\coordinate (O) at (0.5283,2);
			\coordinate (Z) at (0,4);
			\coordinate (W) at (7,2);
			\coordinate (BL) at (0,0);
			\coordinate (BR) at (7,0);
			\coordinate (ZL) at (-2,4.5714);
			\coordinate (ZR) at (9,1.4285);
			\coordinate (OL) at (2,-0.5714);
			\coordinate (OL1) at (-2,2);
			\coordinate (BL1) at (-2,0);
			\coordinate (BR1) at (9,0);
			\coordinate (OL2) at (-0.5714,2);
			\coordinate (B*) at (6.4286,0);
			\coordinate (A*) at (6.4286,2.1633);
			
			\coordinate (A1) at (0.5283, 3.8491);
			\coordinate (B1) at (0.5283, 0);
			\coordinate (A2) at (1.5450, 3.5586);
			\coordinate (B2) at (1.5450, 0);
			\coordinate (A3) at (2.4850, 3.2900);
			\coordinate (B3) at (2.4850, 0);
			\coordinate (A4) at (3.3541, 3.0417);
			\coordinate (B4) at (3.3541, 0);
			\coordinate (B4half) at (3.3541, 1.5);
			
			\draw[thick] (B1) -- (A1) -- (W) -- (BR) -- cycle;
			\draw[thick] (O) -- (W);
			
			\draw[dotted] (A*) -- (B*) -- (W);
			
			\draw[dashed] (ZL) -- (A1);
			\draw[dashed] (W) -- (ZR);
			
			\draw[dashed] (OL1) -- (O);
			\draw[dashed] (W) -- (9,2);
			
			\draw[dashed] (BL1) -- (BL);
			\draw[dashed] (BR) -- (BR1);
			

			\draw[gray, thick] (B1) -- (A2);
			\draw[gray, thick] (A2) -- (B2);
			\draw[gray, thick] (B2) -- (A3);
			\draw[gray, thick] (A3) -- (B3);
			\draw[gray, thick] (B3) -- (A4);
			
			Half spaces%
			\draw [fill=gray, opacity=0.5] (ZL) -- (ZR) -- (9,3.4285) -- (-2,6.5714);
			\draw [fill=gray, opacity=0.5] (BL1) -- (BR1) -- (9,-2) -- (-2,-2);
			\node at (5.2,3.5) {$[\hat{f}\geq\hat{f}(w)]$};
			\node at (3.5,-1) {$[f\leq -\gamma]$};
			
			\draw[gray, thick, dashed] (A4) -- (B4half);
			
			\draw pic[draw, angle radius=1.2cm, angle eccentricity=1.8] {angle=Z--W--O};

			\node[left, yshift=-0.3cm, xshift=-0.1cm] at (O) {$a_0=0$};
			\node[above] at (W) {$w$};
			\node[below, xshift=-0.2cm] at (B1) {$-\gamma u=b_1$};
			\node[below, xshift=0.5cm] at (BR) {$w-\gamma u$};
			\node[above, xshift=0.5cm, yshift=-0.1cm] at (ZL) {$\mathbf{t}$};
			\node[above, xshift=0.5cm] at (OL1) {$[f=0]$};
			\node[above, xshift=0.5cm] at (BL1) {$\mathbf{h}$};
			\node[above] at (A*) {$a^*$};
			\node[below] at (B*) {$b^*$};
						\node[below] at (B2) {$b_2$};

            \node[above] at (A1) {$z$};
            \node[above] at (A3) {$a_2$};
			\node[anchor=west, yshift=0.15cm, xshift=-0.2cm] at ($(A2)+(0,0.15cm)$) {$a_1$};
            \node at (5.3,2.225) {$\theta$};
			
			\end{tikzpicture}
		\end{center}
	\caption{Proof of Lemma~\ref{lemma: geometrical fact}}\label{figure: projections}
	\end{figure}

	\begin{proof}
    Let us denote by $\mathbf{h}$ the  half-line with origin $w-\gamma u$ and passing through the point $-\gamma u$, and by $\mathbf{t}$ the line passing   through  $z$ and $w$. Then denote by $b^*$ the unique point on $\mathbf{h}$  such that $P_{\mathbf{t}}(b^*)=w$. Let us also denote by $a^*$ the unique point on $\mathbf{t}$ such that $b^*=P_{\mathbf{h}}(a^*)$.
    If we define $\theta=\arctan\alpha$, since  $$\frac{\gamma}{\|b^*-w\|}=\cos\theta = \frac{1}{\sqrt{1+\alpha^2}} = \frac{\|w\| }{\sqrt{\|w\|^2+\|z\|^2}} = \frac{\|w\|}{\|w-z\|},$$ 
    we obtain
        \begin{eqnarray*}
         \| a^* -w\| = \underbrace{\left( \frac{\|z\|}{\|w\|} \right)}_{\alpha:=\tan\theta} \| b^* - w\|=\alpha\,\gamma \,\, \frac{\|w-z\|}{\|w\|}. 
        \end{eqnarray*}
 Note that if $b^*\not\in[-\gamma u, w-\gamma u]$ then we necessarily have $z\in[a^*,w]$ and hence
    $$\|a_0-w\|=\|w\|\leq\|z-w\|\leq\|a^*-w\|=\gamma\,\alpha\, \frac{\|w-z\|}{\|w\|},$$
    and our conclusion holds with $m=0$. Hence, we can suppose without loss of generality that $b^*\in[-\gamma u, w-\gamma u]$. \smallskip\newline
 Observe that $$P_{[f\leq-\gamma]}(z)=P_{[f=-\gamma]}(z) =-\gamma u\quad \text{and}\quad P_{[f\leq-\gamma]}(w)=P_{[f=-\gamma]}(w) =w-\gamma u.$$ Since the projection to the affine subspace $[f=-\gamma]$ is an affine function (Fact~\ref{fact: carabinieri per proiezioni}($a$)), for every $x\in [z,w]$, we have $P_{[f\leq-\gamma]}(x)\in [-\gamma u, w-\gamma u]$. Therefore, we infer from Fact~\ref{fact: carabinieri per proiezioni}(b) that, for every $x\in [z,w]$, $P_{\widehat B}(x)\in [-\gamma u, w-\gamma u]$. A similar argument shows that, for every $x\in [-\gamma u, b^*]$, we have $P_{\widehat A}(x)\in [z,w]$.\medskip\newline
     We now consider two cases: \smallskip\newline
    (i) If $\gamma =0$, then $w=b^*\in \widehat{A}\cap \widehat{B}$. Then, setting $\tilde{A}:=[z,w]$ and $\tilde{B}:=[0,w]$ we deduce that 
	$a_n:=\IPA{\widehat{A}}{\widehat{B}}{a_0}{n}=\IPA{\tilde{A}}{\tilde{B}}{a_0}{n}$.
	Since $\IPA{\tilde{A}}{\tilde{B}}{a_0}{n}\rightarrow w$ as $ n\to \infty$ (see, e.g., \cite[Fact~1.2]{BauschkeBorwein93}), the assertion follows.
\medskip\newline
    (ii) If $\gamma\neq 0$ then $b^*\neq w-\gamma u$, that is, we are in the situation described in Figure~\ref{figure: projections}. Notice that,  as long as the elements $b_n:=P_{\widehat B}(a_{n-1})$ are not in the segment $( b^*, w-\gamma u]$, their projections $a_n:=P_{\widehat A}(b_n)$ are in the segment $[z,w]$, and these projections lie in the two-dimensional subspace determined by the points $0,z,w$. 
 Now, let {$m\in\N\cup\{0\}$} be the first non-negative integer such that  $b_{m+1}:=P_{\widehat B}(a_{m}) \in [b^*,w-\gamma u]$.
    Then we easily get 
    $$\| a_m - w\| \leq \| a^* -w\|= \gamma\,\alpha\, \frac{\|w-z\|}{\|w\|},$$
the proof is complete.  \end{proof}

   \textit{Remark.} Figure~\ref{figure: projections} and its role in the proof of Lemma~\ref{lemma: geometrical fact} were inspired by \cite[Fig.~3]{Kop08}.\medskip
	
	\noindent We shall now consider two nonempty closed convex sets $A$ and $B$ that are separated by a functional and show that, provided one of them is bounded, we can perturb the sets to obtain a setting in which the previous lemma can be applied.
	
	\begin{lemma}[Steering the iterations]\label{lemma:  black box}
		Let $A,B$ be closed nonempty convex subsets of $\H$ such that for some $f\in \mathbb{S}_{\H^*}$, $\beta\in \R$, $\gamma\geq0$,  and $r\geq 1$ we have:
		\[
		\sup f(B)\leq \beta -\gamma\leq\beta\leq \inf f(A) \qquad \text{and} \qquad A\subset r \mathbb{B}_{\H}\,.
		\]
		Let further $p,w\in [f=\beta]$, $p\neq w$ and suppose that for some $\delta>0$
		\[
		\left( p+\delta\mathbb{B}_{\H}\right)\cap A\neq \emptyset,\quad \left( w+\delta\mathbb{B}_{\H}\right)\cap A\neq \emptyset,\quad \left( p-\gamma u+\delta\mathbb{B}_{\H}\right)\cap B\neq \emptyset,\quad \left( w-\gamma u+\delta\mathbb{B}_{\H}\right)\cap B\neq \emptyset,
		\]
        where $u\in \mathbb{S}_\H$ represents the functional $f$.\smallskip\newline
		Then there exist $m\in \N\cup\{0\}$ and  closed convex sets $\widehat{A},\widehat{B}$ such that 
        $$\mathrm{D}_{\mathrm{H}}(A,\widehat{A})\leq 3\delta\,,\quad \mathrm{D}_{\mathrm{H}}(B,\widehat{B})\leq 3\delta $$
        and such that, for $a_0:=p$,\, $a_n:= \IPA{\widehat{A}}{\widehat{B}}{a_0}{n}$ and $b_{n}:= P_{\widehat B}(a_{n-1})$ ($n\in\N$), we have
        \begin{enumerate}
            \item the finite sequence $\left\{b_{n}\right\}_{n= 1}^{m+1}$ is contained in $[f=\beta - \gamma]$;
            \item $ a_{m} :=  \IPA{\widehat{A}}{\widehat{B}}{a_0}{m} \in \left( w+\delta \mathbb{B}_{\H} \right) \cap \widehat A.$ 
        \end{enumerate}
        \end{lemma}
           
  \begin{figure}[!ht]
		\begin{center}
			\begin{tikzpicture}[scale=1]
			
			\fill [gray, opacity=0.5] (0,0)--(-2,2)--(6,2)--(4,0);
			\draw[black,thick] (0,0) -- (4,0);
			\draw [fill=gray, opacity=0.5] (-2,2) arc(180:270:2);
			\draw [thick] (-2,2) arc(180:270:2);
			\draw [fill=gray, opacity=0.5] (4,0) arc(270:360:2);
			\draw [thick] (4,0) arc(270:360:2);
			\draw [fill=gray, opacity=0.5] (6,2) arc (0:180:4 and 2);
			\draw [thick] (6,2) arc (0:180:4 and 2);
			\node at (2,2) {$A$};
			
			\draw[gray,thick] (0,-3) -- (4,-3);
			\draw [thick] (-5,2) arc(180:270:5);
			\draw [thick] (4,-3) arc(270:360:5);
			\draw [thick] (9,2) arc (0:180:7 and 5);
			
			\draw[gray,dashed] (-7,-0.5) -- (11,-0.5);
			\node at (-6.5,-1){$[f=\beta=0]$};
			
			\draw [dashed] plot [domain=-7:11] (\x, {-0.3*\x+1});
			\node at (-6.5,2.2){$[\hat{f}=\hat{f}(w)]$};
			
			\draw[dashed,thin] (-0.3,-0.5) arc(0:360:1);
			\node[left,below] at (-1.3,-0.5) {\small{$p=0$}};
			\draw [<->, latex-latex] (-1.3,-0.5) -- (-2.2,0);
			\node at (-1.7,0){\small{$\delta$}};
			
			\draw[dashed,thin] (6,-0.5) arc(0:360:1);
			\node[left,above] at (5,-0.5) {\small{$w$}};
			\draw [<->, latex-latex] (5,-0.5) -- (4.1,-1);
			\node at (4.6,-1){\small{$\delta$}};

			\draw[black,thick] (-0.2,-4) -- (3.8,-4);
			\draw [thick] (-0.2,-4) arc(90:180:5);
			\draw [thick] (3.8,-4) arc(-270:-360:5);
			
			\fill[gray,opacity=0.5] (-0.2,-7)--(3.8,-7)--(5.8,-9)--(-2.2,-9);
			\draw[black,thick] (-0.2,-7) -- (3.8,-7);
			\draw[fill=gray, opacity=0.5] (-0.2,-7) arc(90:180:2);
			\draw [thick] (-0.2,-7) arc(90:180:2);
			\draw[fill=gray, opacity=0.5] (3.8,-7) arc(-270:-360:2); 
			\draw [thick] (3.8,-7) arc(-270:-360:2);
			\node at (2,-8) {$B$};
			
			\draw[gray,dashed] (-7,-6.5) -- (11,-6.5);
			\node at (-6.5,-7){$[f=-\gamma]$};
			
			\draw[gray,dashed] (-1.3,1.24) -- (-1.3,-6.5);
			
			\draw[dashed,thin] (-0.3,-6.5) arc(0:360:1);
			\node[below] at (-1.3,-6.5) {\small{$-\gamma u$}};
			\draw [<->, latex-latex] (-1.3,-6.5) -- (-2.2,-6);
			\node at (-1.7,-6){\small{$\delta$}};
			
			\draw[gray,dashed] (5,-0.5) -- (5,-6.5);
			
			\draw[dashed,thin] (6,-6.5) arc(0:360:1);
			\node[below] at (5,-6.5) {\small{$w-\gamma u$}};
			\draw [<->, latex-latex] (5,-6.5) -- (5.9,-6);
			\node at (5.4,-6){\small{$\delta$}};
			
			\pattern[pattern=checkerboard light gray,opacity=0.1](-4.55,-6.5) -- (8.15,-6.5) -- (8.15,-9) -- (-4.55,-9);
			\draw[black,ultra thick] (-4.55,-6.5) -- (8.15,-6.5);
			\pattern[pattern=checkerboard light gray,opacity=0.1](-4.55,-6.5) arc(150:180:5);
			\pattern[pattern=checkerboard light gray,opacity=0.1](-4.55,-6.5) -- (-4.55,-9) -- (-5.2,-9);
			\draw [ultra thick] (-4.55,-6.5) arc(150:180:5);
			\pattern[pattern=checkerboard light gray,opacity=0.1](8.15,-6.5) arc(-330:-360:5);
			\pattern[pattern=checkerboard light gray,opacity=0.1](8.15,-6.5) -- (8.15,-9) -- (8.8,-9);
			\draw [ultra thick] (8.15,-6.5) arc(-330:-360:5);
			\node at (-3,-8) {$\widehat{B}$};

			\pattern[pattern=checkerboard light gray,opacity=0.1](9,2) arc (0:175:7 and 5);
			\draw [ultra thick] (9,2) arc (0:175:7 and 5);
			\pattern[pattern=checkerboard light gray,opacity=0.1](9,2) arc(0:-42:5);
			\draw [ultra thick] (9,2) arc(0:-42:5);
			\pattern[pattern=checkerboard light gray,opacity=0.1](9,2) -- (7.7,-1.35) -- (-5.127,2.465);
			\draw [ultra thick] plot [domain=-5:7.75] (\x, {-0.3*\x+1});
			\node at (-3,4) {$\widehat{A}$};
			
			\draw [<->, latex-latex] (6,2) -- (9,2);
			\node at (7.5,2.5){$3\delta$};
			\draw [<->, latex-latex] (2,-7) -- (2,-4);
			\node at (2.5,-5.5){$3\delta$};		
			\end{tikzpicture}
		\end{center}
		\caption{Construction of $\widehat{A}$ and $\widehat{B}$}\label{blackbox}
	\end{figure}
            
	\begin{proof}
		Without loss of generality, we may assume that $p=0$ and consequently $\beta=0$, $\langle u, w\rangle =0$, $w\in \ker f$ and $A\subset [f\geq 0]$. We set:
		\[
		\widehat{B}=\left(B+3\delta \mathbb{B}_{\H}\right)\cap \left[f\leq -\gamma\right].
		\] Since $\mathrm{D}_{\mathrm{H}}(B,B+3\delta \mathbb{B}_{\H})\leq 3\delta$ and $B\subseteq [f\leq -\gamma]$, we have $B\subset\widehat{B}\subset B+3\delta \mathbb{B}_{\H}$ and therefore
		\[
		\mathrm{D}_{\mathrm{H}}(B,\widehat{B})\leq 3\delta.
		\]
		Set a positive number $\epsilon_0$ so that 
        $$\varepsilon_0\leq\min\left\{1,\|w\|\right\} \qquad \text{(the exact value of $\varepsilon_0$ will be fixed later)}.$$
        Let $u\in \mathbb{S}_\H$ representing the functional $f\in \mathbb{S}_{\H^*}$ (that is, $f(\cdot)=\langle u,\cdot\rangle$) and define: 
		\begin{equation}\label{eq:theta}
			\widehat{f}(x)= \langle u+\alpha\frac{w}{\|w\|} ,x\rangle = f(x)+\alpha\langle \frac{w}{\|w\|},x\rangle,\quad x\in\H, \qquad\text{where } \quad \alpha=\frac{\delta \varepsilon_0}{r\|w\|}.
		\end{equation}
		Since $\|u\|=1$ and $\langle u,w\rangle=0$, we have $\|\widehat{f}\|>1$. Moreover, for every $a\in A$ we have: 
		\begin{equation}\label{eq:jpp}
			\widehat{f}(a)=f(a)+\alpha \langle \frac{w}{\|w\|},a\rangle \geq -\alpha \|a\|.
		\end{equation}
		We now consider the set 
		$$
		\widehat{A}=\left(A+3\delta \mathbb{B}_{\H}\right)\,\bigcap \,\left[ \widehat{f}\geq \widehat{f}(w)\right].
		$$
		\textit{Claim}. $\mathrm{D}_{\mathrm{H}}(A,\widehat{A})\leq 3\delta$. \medskip\newline        
		\textit{Proof of the claim.} Since $\widehat{A}\subset A+3\delta \mathbb{B}_{\H}$, we have $e(\widehat{A},A)\leq 3\delta$. Let now $a\in A$. If $a \in \left[\widehat{f}\geq \widehat{f}(w)\right]$, then $a\in\widehat{A}$, while if $\widehat{f}(a)<\widehat{f}(w)=\frac{\delta\,\varepsilon_0}{r}=\alpha \|w\|$, we set $$\tilde{a}=P_{\left[\widehat{f}\geq \widehat{f}(w)\right]}(a)=P_{\left[\widehat{f}= \widehat{f}(w)\right]}(a).$$ Then $\widehat{f}(\tilde{a})=\widehat{f}(w)= \delta \varepsilon_0 r^{-1}$. 
        Recalling that $\|a\| \leq r$, for all $a\in A$ and $\varepsilon_0\leq 1\leq r$, we deduce from~\eqref{eq:jpp} that 
		\[
		\left\|\Tilde{a}-a\right\|=\dfrac{\left|\widehat{f}(w)-\widehat{f}(a)\right|}{\|\widehat{f}\|} < \widehat{f}(w)-\widehat{f}(a) \leq \delta\,\left(\frac{\varepsilon_0}{r}\right) + \delta\,\varepsilon_0\,\left(\frac{\|a\|}{r}\right) \leq 2\,\delta.
		\]        
		It follows that $\tilde{a}\in\widehat{A}$, whence
		$e(A,\widehat{A})\leq 2\delta$ and the claim follows. \medskip \newline
		We finally set 
        \begin{equation} \label{eq:tahar}
            z:=\left(\frac{\delta \varepsilon_0}{r}\right)u = \alpha \|w\|\,u\,.
        \end{equation}
  By the definition of the sets $\widehat{A}$ and $\widehat{B}$,  we deduce that 
		\[ \left[z,w\right]\subset \widehat{A}
\qquad\text{and} \qquad 		[-\gamma u,w-\gamma u]\subset \widehat{B} .
		\]
		Therefore, we can apply Lemma~\ref{lemma: geometrical fact} for the sets $\widehat{A}, \widehat{B}$ 
        (note that $\alpha$ is indeed equal to $\frac{\|z\|}{\|w\|}$)
		to conclude that 
		there exists  $m\in\N$ satisfying 
        	$$\|\underbrace{\IPA{\widehat A}{\widehat B}{a_0}{m}}_{a_m}-w\| \leq \,\gamma\,\frac{\|z\|}{\|w\|}\, \frac{\|w-z\|}{\|w\|} \leq\left(\gamma \frac{\|w\|+\|z\|}{\|w\|^2}\,\frac{\varepsilon_0}{r}\,\right)\delta\,.$$ 
    Shrinking the value of $\varepsilon_0$ and recalling~\eqref{eq:tahar} we ensure that the above quantity is less or equal to $\delta$. The proof is complete.
	\end{proof}
	\medskip
    
	\noindent We now establish our main result under the additional assumption that the sets $A$ and $B$ are disjoint (and the best approximation sets $E$ and $F$ are nonempty). This assumption considerably simplifies the proof, allowing us to outline the geometrical features of the argument. \smallskip\newline
	Given a pair $(A,B)$ of nonempty closed convex sets in $\H$, we recall from~\eqref{eq:E,F} the definition of the best approximation sets $E, F$ and from~\eqref{eq:displ} the definition of the  displacement vector $v:=P_{\overline{B-A}}(0)$. We start with the following technical lemma.
	\begin{lemma}[lack of regularity]\label{lem: separated case}
		Let $A,B$ be two closed convex sets in~$\H$. Assume that $E,F$ are nonempty and disjoint and that the pair $(A,B)$ is not regular. \newline Then there exist $\varepsilon_0>0$, a sequence $\{\delta_n\}_{n\geq 1}\subset (0,+\infty)$ with $\delta_n\to 0$, and a sequence $\{w_n\}_{n\geq 1} \subset\H$ such that for every $n\in\N$ we have:
        \begin{enumerate}
            \item $	\varepsilon_0 < \dist(w_n, E) \leq  \varepsilon_0 +1 $; \item
            $\max\Big\{ \dist(w_n, A), \,\dist(w_n, B-v) \Big\}\,\leq \delta_n$;
            \item $\langle v, w_n\rangle=\langle v, \bar a \rangle$ for some (equivalently, for all) $\bar a\in E$.
        \end{enumerate}
	\end{lemma}

    \begin{proof}
		Since the pair $(A,B)$ is not regular, there exist 
		$\{w_n'\}_{n\geq 1}\subset \H$ and 
		$\varepsilon_0>0$ such that 
		\[
		\dist(w_n', E)> 3\varepsilon_0, \quad\text{for all } n\geq 1 \quad \text{ and } \quad \lim_{n\to\infty} \max\left\{\dist(w_n', A),\,\dist(w_n', B-v)\right\} = 0. 
		\]
        Since $E,F$ are nonempty and $F=E+v$, we can separate $A$ and $B-v$ with an affine hyperplane $H$ orthogonal to the displacement vector $v$, containing the set $E$.  
		Then setting $w_n''\coloneqq P_{H}(w_n')$, one has
		$$
		\|w_n'-w_n''\| \leq \max\{ \dist(w_n', A),\dist(w_n', B-v)\},
		$$
		yielding
		$$
		\lim_{n\to\infty} \|w_n'-w_n''\| =0\qquad\text{and}\qquad  \lim_{n\to\infty} \max\left\{\dist(w_n'', A),\,\dist(w_n'', B-v)\right\} = 0. 
		$$
		Then for some $n_0\in\N$ sufficiently large and all $n\ge n_0$ we have:
		\[
		\dist(w_n'', E) \geq \dist(w_n', E) - \|w_n'-w_n''\| >2\varepsilon_0\,.
		\]
		We now pick $w_n\in [w_n'', P_{E}(w_n'')]$ so that 
		$\varepsilon_0<\dist(w_n,E)< \varepsilon_0+1$. Finally, by convexity of the sets $A$ and $B$ we have 
		$$
		\delta_n:=\max\{ \dist(w_n, A),\dist(w_n, B-v)\}\leq  \max\left\{\dist(w_n'', A),\,\dist(w_n'', B-v)\right\}
		$$
		and the proof is complete.
	\end{proof}
	\medskip
	
	We are now ready to establish the result in the special case of the aforementioned assumption that the best approximation sets are (nonempty and) disjoint.
	
	\begin{theorem}[main result: case $A\cap B = \emptyset$]\label{th: separated case}
		Let $A,B$ be nonempty closed convex subsets of $\H$ such that $E,F$ are nonempty and disjoint. 
		If the pair $(A,B)$ is not regular, then $(A,B)$ is not $d$-stable.   
	\end{theorem}
  
\begin{proof}
Notice that $v\neq 0$ (displacement vector) and that we can assume that for any initial point $x$, the alternating projection sequence with respect to the sets $A$, $B$ satisfies
		\begin{equation}\label{eq:ar}
			\lim_{n\to\infty}\,\dist\left(\IPA{A}{B}{x}{n},\, E \right) = 0,
		\end{equation}
because, otherwise, the pair $(A,B)$ will already be not $d$-stable. Moreover, there is no loss of generality to assume $0\in E$. Let $f\in \H^*$ be the linear functional defined by $f(\cdot)\coloneqq \langle \cdot,-v\rangle$. \smallskip\newline
By Lemma~\ref{lem: separated case} (lack of regularity), there exist $\varepsilon_0>0$, a sequence of positive numbers $\{\delta_n\}_n$ with $\delta_n\to 0$ and a sequence 
		$\{w_n\}_{n\geq 1}\subset \ker f$ such that 
		\[
		\dist(w_n,E)> 3\varepsilon_0 \qquad \text{and}\qquad \max\Big\{ \dist(w_n, A), \,\dist(w_n, B-v) \Big\} \,\leq\,\delta_n \,\underset{n\to\infty}{\longrightarrow} \,0\,.
		\]
		We can 
		assume that $\delta_n\leq \varepsilon_0$, for all $n\in \N$. \smallskip\newline
        Take an arbitrary $q_0\in\H$: in view of ~\eqref{eq:ar}, there exist $n_1 \in \N$ and $p_1 \in E$, both depending on $\varepsilon_0$, such that    
\begin{equation} \label{eq:mp}
		\tilde{p}_1\coloneqq \IPA{A}{B}{q_0}{n_1} \in \left(p_1+\varepsilon_0\mathbb{B}_{\H}\right) \cap A\,. 
\end{equation}
\noindent We shall now show that for any $r_1>0$ sufficiently large, there exist two closed convex sets  $\widehat{A}_1$ and $\widehat{B}_1$ which are $3\delta_1$-close  (with respect to the Hausdorff distance ${\mathrm{D_H}}$) to the sets $A\cap r_1\mathbb{B}_{\H}$ and $B$, respectively, such that the iterations of the corresponding alternating projection method, starting from $\tilde{p_1}$, bring us to some point $q_1$ sufficiently close to the point $w_1$. \smallskip\newline        
        To this end, we first show that starting from the point $p_1\in \left( \tilde{p_1}+\varepsilon_0\mathbb{B}_{\H}\right)\cap \ker f$
we can get sufficiently close to the point $w_1$: indeed, notice that $p_1\in E\subset \ker f$ and $w_1 \in \ker f$. Moreover, the set $w_1+\delta_1\mathbb{B}_{\H}$ has nontrivial intersection with both sets $A$ and $B-v$. Let  $\tilde{a}_1\in A$ be such that $\|\tilde{a}_1-w_1\|\leq \delta_1$, and pick $r_1 \geq \max\{1,\|\tilde{a}_1\|,\|p_1\|\}$.\smallskip\newline
        Therefore, we can apply Lemma~\ref{lemma:  black box} to the sets $A\cap r_1\mathbb{B}_{\H}$ and $B$, for $$\delta\coloneqq \delta_1 \leq \varepsilon_0,\quad r\coloneqq r_1,\quad \gamma\coloneqq \|v\|,\quad \beta\coloneqq 0$$
        to obtain $m_1\in\N\cup\{0\}$  and closed convex sets $\widehat{A}_1$ and $\widehat{B}_1$ such that $$\mathrm{D}_{\mathrm{H}}(A\,\cap\, r_1\mathbb{B}_{\H},\widehat{A}_1)\leq 3\delta_1, \qquad \mathrm{D}_{\mathrm{H}}(B,\widehat{B}_1)\leq 3\delta_1$$ and 
        \[
		\tilde{w}_1\coloneqq\IPA{\widehat{A}_1}{\widehat{B}_1}{p_1}{m_1}\in w_1+\delta_1\mathbb{B}_{\H}\subset w_1+\varepsilon_0\mathbb{B}_{\H}.
		\]
Starting now the iterations from the point $\tilde{p_1}$ 
satisfying~\eqref{eq:mp}, since $\|p_1-\tilde{p_1}\| \leq \varepsilon_0$ and the projection operators $P_{\widehat{A}_1}, P_{\widehat{B}_1}$ are non-expansive, we deduce
		\[
		q_1\coloneqq \IPA{\widehat{A}_1}{\widehat{B}_1}{\tilde{p}_1}{m_1} \,\in \tilde{w}_1+\varepsilon_0\mathbb{B}_{\H}\,.
		\]
		Consequently, $\|q_1-w_1 \| \leq 2\varepsilon_0$. Notice that since $\dist(w_1,E) > 3 \varepsilon_0$
		we get $\dist(q_1,E) > \varepsilon_0$.
		\smallskip  \newline
		Concatenating the sequences:
		\[
		\{\IPA{A}{B}{q_0}{n}\}_{n\leq n_1} \quad(\text{joining $q_0$ to $\tilde{p_1}$}) \qquad\text{and}\qquad \{\IPA{\widehat{A}_1}{\widehat{B}_1}{\tilde{p_1}}{n}\}_{n\leq m_1} \quad(\text{joining $\tilde{p}_1$ to $q_1$}) 
		\]
		we obtain a finite sequence of variational alternating projections whose last point $q_1$ is outside the set $E +\varepsilon_0\mathbb{B}_{\H}$.
		\smallskip\newline
        Iterating the same argument as above, we can inductively obtain a sequence of sets $\widehat{A}_h,\widehat{B}_h\subset \H$, together with points $q_h,\tilde{p}_h\in \H$,  $p_h\in E$, $ \tilde{a}_h\in A$, positive numbers $r_h$, and non-negative integers $n_h,m_h$ such that, for every  $h\in \N$, we have:
		\begin{enumerate}
			\item $\tilde{p}_h=\IPA{A}{B}{q_{h-1}}{n_h}$ and $\|p_h-\tilde{p}_h\|<\epsilon_0$;\smallskip
			\item $\|\tilde{a}_h-w_h\|\leq \delta_h$ and $\max\{\|p_h\|,\|\tilde{a}_h\|,h\}\leq r_h$,\smallskip
			\item\label{HausDist_A} $ \mathrm{D}_{\mathrm{H}}\left(A\cap r_h \mathbb{B}_{\H},\widehat{A}_h\right)\leq 3\delta_h$ and $\mathrm{D}_{\mathrm{H}}\left(B,\widehat{B}_h\right)\leq 3\delta_h$;\smallskip
			\item\label{iteration_A} $q_{h}=\IPA{\widehat{A}_h}{\widehat{B}_h}{\tilde{p}_h}{m_h}=\IPA{\widehat{A}_h}{\widehat{B}_h}{\IPA{A}{B}{q_{h-1}}{n_h}}{m_h}$ and $\dist(q_h,E) \geq \varepsilon_0$. 
		\end{enumerate}
        To prove that the pair $(A,B)$ is not 
		$d$-stable, let us consider the sequence $\{A_k\}_k$ of closed convex sets in $\H$ defined by  
		$$\left\lbrace\underbrace{A,\ldots,A}_{n_1\text{ times}},\underbrace{\widehat{A}_1,\ldots,\widehat{A}_1}_{m_1\text{ times}},\underbrace{A,\ldots,A}_{n_2\text{ times}},\underbrace{\widehat{A}_2,\ldots,\widehat{A}_2}_{m_2\text{ times}},\underbrace{A,\ldots,A}_{n_3\text{ times}},\underbrace{\widehat{A}_3,\ldots,\widehat{A}_3}_{m_3\text{ times}},\ldots\right\rbrace$$
		and the sequence $\{B_k\}_k$ of closed convex sets in $\H$ defined by the same process: 
		$$\left\lbrace\underbrace{B,\ldots,B}_{n_1\text{ times}},\underbrace{\widehat{B}_1,\ldots,\widehat{B}_1}_{m_1\text{ times}},\underbrace{B,\ldots,B}_{n_2\text{ times}},\underbrace{\widehat{B}_2,\ldots,\widehat{B}_2}_{m_2\text{ times}},\underbrace{B,\ldots,B}_{n_3\text{ times}},\underbrace{\widehat{B}_3,\ldots,\widehat{B}_3}_{m_3\text{ times}},\ldots\right\rbrace$$
		It follows from~\eqref{HausDist_A} and Fact~\ref{fact: AW-intersection balls} that the sequences $\{\widehat{A}_h\}_h$ and $\{\widehat{B}_h\}_h$ Attouch-Wets converge to the sets $A$ and $B$ respectively. 
		Therefore, so do the sequences $\{{A}_k\}_k$ and $\{{B}_k\}_k$ (in which finite segments of constant sequences $A_k=A$ and $B_k=B$ respectively, are interposed). Consequently, considering the perturbed alternating projection sequences $\{a_h\}$ and $\{b_h\}$, with respect to $\{A_h\}_{h}$ and $\{B_h\}_{h}$ and with starting point $q_0$, we infer from~\eqref{iteration_A} that:
		\[
		a_{k_N}=q_{N+1} \qquad\text{where} \quad k_N=\sum_{h=1}^N (n_h+m_h).
		\]
		Since $k_N\to \infty$ as $N\to\infty$ and $\dist\left(q_{N+1},E\right)\geq \epsilon_0$, we deduce that the pair $(A,B)$ is not $d$-stable. The proof is complete.		
	\end{proof}
	\medskip
	
	
	\noindent We shall now establish the main result of our paper in the case where $A \cap B$ is nonempty. For the proof, we need to recall the notion of algebraic interior of a set. 
	\begin{definition}\label{def: algebraic interior} Let $C\subseteq \H$ be a nonempty set. The {\em algebraic interior} of $C$, denoted by $\mathrm{alg\,int\, } C$, is the set of all points $c\in C$ such that for every $u\in \mathbb{S}_{\H}$ there exists $\varepsilon_u>0$ such that $[c,c+\varepsilon_u u)\subseteq C$.  
	\end{definition}
	
	\noindent We recall that a set $C\subseteq \H$ is called {\em absorbing} if for every $x \in \H $, there exists a positive number $\alpha$ such that $x\in tC$ whenever 
	$t >\alpha$. It is easy to see that $x\in \mathrm{alg\,int\, } C$ if and only if $C-x$ is absorbing. Moreover, if $C$ is a convex set such that $\inte C\neq \emptyset$ then, 
	$\mathrm{alg\,int\, } C=\inte C$. 
	\smallskip\newline
	We also need the following well-known corollary of Baire category theorem.
	
	\begin{fact}[{see, e.g., \cite{AliprantisBorder94}*{Corollary~3.28}}]\label{fact: Baire} If a complete metric space is a countable union of closed sets, then at least one of them has a nonempty interior.  
	\end{fact}
	\begin{theorem}[main result: case $A\cap B \neq \emptyset$]\label{th: general case} 
	Let $A,B$ be closed convex subsets of $\H$ such that $A\cap B$ is nonempty. Suppose that the pair $(A,B)$ is not regular. 
    Then the pair $(A,B)$ is not $d$-stable.	
	\end{theorem}
	
	\begin{proof} By assumption, we have $E=F=A\cap B \neq \emptyset$ and $v=0$ (displacement vector). As in the proof of Theorem~\ref{th: separated case}, without loss of generality we may assume that the alternating projection sequence relative to the sets $A$ and $B$ satisfies~\eqref{eq:ar} for any starting point  $x\in\H$. We can also assume that $0\in A\cap B$. \smallskip\newline
		Since the pair $(A, B)$ is not regular, 
		there exist  
		$\epsilon_0\in\left(0,1\right)$ and sequences $\{\delta_n\}_{n\ge 1} \subset (0, \epsilon_0/6)$ with $\delta_n\to 0 $ and $\{w_n\}_n\subset \mathcal{H}
		$ such that:
		\[
		\dist (w_n, A\cap B)\geq 3\epsilon_0 \qquad \delta_n:=\max\left\{\dist (w_n, A),\dist (w_n, B)\right\}.
		\]
		\noindent Take an arbitrary point $q_0\in \H$. Let $n_1\in \N$ be such that $$\dist\left(\IPA{A}{B}{q_0}{n_1},\,A\cap B\right)<\,\epsilon_0.$$
		Let $p_1\in A\cap B$ be such that 
        $$\tilde{p}_1=\IPA{A}{B}{q_0}{n_1} \in \left(p_1 + \varepsilon_0\mathbb{B}_{\H}\right)\cap A$$
        and let $\tilde{a}_1\in A$ and $\tilde{b}_1\in B$ be such that $\|\tilde{a}_1-w_1\|\leq \delta_1$ and $\|\tilde{b}_1-w_1\|\leq \delta_1$. Choosing $r_1>0$ such that $\max\{\|p_1\|,\|\tilde{a}_1\|,1\}\leq r_1$, let us define the sets  
		\[
		A_1=A\cap r_1\mathbb{B}_{\H} \qquad\text{and }\qquad B_1=B.
		\]
		
		\noindent \textit{Claim.} $0\not\in \mathrm{alg\,int\,}(A_1-B_1).$ 
		\medskip\newline
		\textit{Proof of the claim.} Since the pair $(A,B)$ is not regular, by \cite{BauschkeBorwein93}*{Corollary~4.5}, we have $0\not\in\inte(A-B)$. Therefore, since $A_1-B_1\subset A-B$, we obtain $0\not\in\inte(A_1-B_1)$. Notice further that since $A_1$ is $w$-compact, the set $A_1-B_1$ is closed and convex. Let us suppose, towards a contradiction, that $0\in \mathrm{alg\,int\, }(A_1-B_1)$. Then the set $A_1-B_1$ would be an absorbing set and consequently $\cup_{n=1}^\infty n(A_1-B_1)=\H$. In this setting, Fact~\ref{fact: Baire} would yield that $\inte(A_1-B_1)\neq \emptyset$. Since $(A_1-B_1)$ is convex, we would have $0\in \mathrm{alg\,int\, }(A_1-B_1)=\inte(A_1-B_1)$, a contradiction. Therefore, the assertion of the claim holds. \medskip
		
		\noindent It follows that there exists $u_1\in \mathbb{S}_\H$ such that, for every $\theta>0$, we have $A_1\cap(B_1+\theta u_1)=\emptyset$. In particular, the closed convex sets $A_1$ and $(B_1+\delta_1 u_1)$ are disjoint and $A_1$ is $w$-compact. By the Hahn-Banach theorem there exists  $f_1\in \mathbb{S}_{\H^*}$ and $\beta_1\in\R$ such that
		$$
		\sup  f_1\left(B_1+\delta_1 u_1\right)\leq \beta_1 \leq  \inf  f_1\left(A_1\right).
		$$
		Since $p_1\in A_1$ and $p_1+\delta_1 u_1\in \left(B_1 +\delta_1 u_1\right)$, there exists $s_1\in [p_1,p_1+\delta_1 u_1 ]$ such that $f_1(s_1)=\beta_1$. 
		It follows that $$(s_1+\delta_1 \mathbb{B}_{\H})\cap A_1\neq\emptyset,\qquad (s_1+\delta_1\mathbb{B}_{\H})\cap \left(B_1+\delta_1u_1\right)\neq\emptyset.$$
		
		\noindent Similarly, 
		since $\tilde{a}_1\in A_1$ and $\tilde{b}_1+\delta_1u_1\in \left(B_1+\delta_1u_1\right)$, there exists $t_1\in [\tilde{a}_1,\tilde{b}_1+\delta_1 u_1 ]$ such that $f_1(t_1)=\beta_1$. Since 
		$\|\tilde{b}_1+\delta_1 u_1-w_1\|\leq 2\delta_1$ and $\|\tilde{a}_1-w_1\|\leq \delta_1$, by convexity we have  $\|w_1- t_1\|\leq2\delta_1$. Recalling that $\dist(w_1, A_1) <\delta_1$ we obtain 
		$$(t_1+3\delta_1\mathbb{B}_{\H})\cap A_1\neq\emptyset\qquad\text{and}\qquad  (t_1+3\delta_1 \mathbb{B}_{\H})\cap (B_1+\delta_1u_1)\neq\emptyset.$$
		Applying  Lemma \ref{lemma:  black box} to the sets $A_1,B_1+\delta_1 u_1$, for
        $$
        \delta\coloneqq \delta_1,\quad r\coloneqq r_1,\quad \gamma\coloneqq 0,\quad \beta\coloneq \beta_1,
        $$
        we obtain $m_1\in\N\cup\{0\}$ and two (perturbed) closed convex sets $\widehat{A}_1,\widehat{B}_1$ such that 
		$$ \mathrm{D}_{\mathrm{H}}\left(A_1,\widehat{A}_1\right)\leq 9\delta_1\quad\text{and}\quad \mathrm{D}_{\mathrm{H}}\left(B_1+\delta_1u_1,\widehat{B}_1\right)\leq 9\delta_1,  
		$$
		and such that $$\|\IPA{\widehat{A}_1}{\widehat{B}_1}{s_1}{m_1} - t_1\|\leq 3\delta_1.$$ Notice that
		$$ \mathrm{D}_{\mathrm{H}}\left(A\cap r_1 \mathbb{B}_{\H},\widehat{A}_1\right)\leq 9\delta_1\quad\text{and}\quad \mathrm{D}_{\mathrm{H}}\left(B,\widehat{B}_1\right)\leq 10\delta_1.  
		$$
		Since 
		$$\dist\left(t_1,A\cap B\right)\geq \dist\left(w_1,A\cap B\right) -\|w_1-t_1\|\geq3\epsilon_0-2\delta_1>2\epsilon_0+4\delta_1,$$
		we have
		$$\dist\left(\IPA{\widehat{A}_1}{\widehat{B}_1}{s_1}{m_1},A\cap B\right)\geq \dist\left(t_1,A\cap B\right)- \|\IPA{\widehat{A}_1}{\widehat{B}_1}{s_1}{m_1} - t_1\|        
        >2\epsilon_0+4\delta_1-3\delta_1=2\epsilon_0+\delta_1.$$
		Since $\|s_1-\tilde{p}_1\|\leq\epsilon_0+\delta_1$, we have $\|\IPA{\widehat{A}_1}{\widehat{B}_1}{\tilde{p}_1}{m_1}-\IPA{\widehat{A}_1}{\widehat{B}_1}{s_1}{m_1}\|\leq \epsilon_0+\delta_1$ and hence
		$$\dist\left(\IPA{\widehat{A}_1}{\widehat{B}_1}{\tilde{p}_1}{m_1},A\cap B\right)\geq\dist\left(\IPA{\widehat{A}_1}{\widehat{B}_1}{s_1}{m_1},A\cap B\right)-(\epsilon_0+\delta_1)\geq\epsilon_0.$$
		Let us finally set 
		$$q_1=\IPA{\widehat{A}_1}{\widehat{B}_1}{\tilde{p}_1}{m_1}=\IPA{\widehat{A}_1}{\widehat{B}_1}{\IPA{A}{B}{q_0}{n_1}}{m_1}.$$
		\smallskip 
		Iterating the same argument as above, we can inductively obtain a sequence of sets $\widehat{A}_h,\widehat{B}_h\subset \H$, together with points $q_h,\tilde{p}_h\in \H$,  $p_h\in A\cap B$, $ \tilde{a}_h\in A$, $\tilde{b}_h\in B$,  positive numbers $r_h$ and positive integers $n_h,m_h$ such that for every  $h\in \N$, we have:
		\begin{enumerate}
			\item $\tilde{p}_h=\IPA{A}{B}{q_{h-1}}{n_h}$ and $\|\tilde{p}_h-p_h\|<\epsilon_0$;\smallskip
			\item $\|\tilde{a}_h-w_h\|\leq \delta_h$, $\|\tilde{b}_h-w_h\|\leq \delta_h$ and $\max\{\|p_h\|,\|\tilde{a}_h\|,h\}\leq r_h$, \smallskip
			\item\label{HausDist} $ \mathrm{D}_{\mathrm{H}}\left(A\cap r_h \mathbb{B}_{\H},\widehat{A}_h\right)\leq 9\delta_h$ and $\mathrm{D}_{\mathrm{H}}\left(B,\widehat{B}_h\right)\leq 10\delta_h$;\smallskip
			\item\label{iteration} $q_{h}=\IPA{\widehat{A}_h}{\widehat{B}_h}{\tilde{p}_h}{m_h}=\IPA{\widehat{A}_h}{\widehat{B}_h}{\IPA{A}{B}{q_{h-1}}{n_h}}{m_h}$  and $\dist(q_h,A\cap B) \geq \varepsilon_0$. 
		\end{enumerate}
        \smallskip
        
        Similarly to the last part of the proof of Theorem~\ref{th: separated case}, we generate two sequences $\{A_n\}$ and $\{B_n\}$ of closed convex sets of the form 
		$$\left\lbrace A_n\right\rbrace_{n=1}^\infty:=
		\left\lbrace\underbrace{A,\ldots,A}_{n_1\text{ times}},\underbrace{\widehat{A}_1,\ldots,\widehat{A}_1}_{m_1\text{ times}},\underbrace{A,\ldots,A}_{n_2\text{ times}},\underbrace{\widehat{A}_2,\ldots,\widehat{A}_2}_{m_2\text{ times}},\underbrace{A,\ldots,A}_{n_3\text{ times}},\underbrace{\widehat{A}_3,\ldots,\widehat{A}_3}_{m_3\text{ times}},\ldots\right\rbrace
		$$
		$$
		\left\lbrace B_n\right\rbrace_{n=1}^\infty:= \left\lbrace\underbrace{B,\ldots,B}_{n_1\text{ times}},\underbrace{\widehat{B}_1,\ldots,\widehat{B}_1}_{m_1\text{ times}},\underbrace{B,\ldots,B}_{n_2\text{ times}},\underbrace{\widehat{B}_2,\ldots,\widehat{B}_2}_{m_2\text{ times}},\underbrace{B,\ldots,B}_{n_3\text{ times}},\underbrace{\widehat{B}_3,\ldots,\widehat{B}_3}_{m_3\text{ times}},\ldots\right\rbrace
		$$
		respectively, for which the variational alternating projection sequences starting from $q_0$ and successively projecting onto these sets do not satisfy the conditions in Definition~\ref{def:perturbedseq}. This shows that the pair $(A,B)$ is not $d$-stable. 
	\end{proof}
	The above result provides a complete proof of Theorem~A and positively answers the open problem considered in~Problem~\ref{open problem}.

	\section{Open Problems and Remarks}\label{sec: problems}
	\noindent 
	A typical example of convex feasibility problem where regularity fails is the so called Moment Problem. This problem consists of a pair $(A,B)$ in a Hilbert space $\mathcal{H}$ given by a linear subspace~$A$ of finite codimension~$k$ and a Hilbert lattice cone $B$. This problem has an intrinsic applied nature, as the particular case of $B$ being the cone of positive elements corresponds to the research of positive zeros of a linear functional operator via von Neumann's projection algorithm.\smallskip\newline
	Bauschke and Borwein \cite{BauschkeBorwein93} proved that, for the case of codimension $k=1$, the alternating projection sequences always converge to a point in $A\cap B$; at the same time, they provided an example (see \cite{BauschkeBorwein93}*{Example~5.5}) where the pair $(A,B)$ has unbounded intersection and is not regular, and so is not $d$-stable thanks to Theorem~\ref{main theorem}. In particular, this means that there exists a sequence of Attouch-Wets approximations for $A$ and $B$ for which the corresponding variational projection sequences fails to converge: however, as we have seen, an explicit construction of  such sequences is a difficult problem.
	\smallskip\newline
	\textit{A posteriori}, the fact that $d$-stability does not hold for these instances of Moment Problem is not surprising since the involved lattice cone $B$ has usually an unbounded basis (and the approximation of such cones is often problematic). For the time being, it remains unknown if the alternating projection sequences converges in norm for a general moment problem of codimension $k>1$. Our results suggest that the variational approach does not seem to be the proper way to investigate the Moment Problem and different techniques related to the lattice structure of the cone are required, especially for higher codimension $k$, see \cite{BattistoniMiglierina}. \smallskip\newline
	To summarize, we hereby established the equivalence between regularity and $d$-stability for a pair of closed convex subsets $(A,B)$ of $\H$ provided that $A\cap B$ is nonempty and bounded (or more generally, if the best approximation sets $E$ and $F$ are nonempty and bounded). Example~5.2 in \cite{DebeMiglStab} shows that regularity does not imply $d$-stability of the pair $(A,B)$ when $A\cap B$ is unbounded, even in a finite-dimensional setting. \smallskip\newline
    At this point, let us recall from \cite[Theorem~3.7]{BD2022} that the alternating projection sequence $a_n=P_AP_B(a_{n-1})$, $n\geq 1$, is \textit{self-contracted}, that is, for any $N\geq 1$ the finite sequence $\{\|a_i-a_N\|\}_{1\leq i\leq N}$ is decreasing. This fact guaranties that the set of accumulation points of the alternate projection sequence $\{a_n\}_{n\geq 1}$ is at most a singleton (in particular, the method converges in finite dimensions, whenever the best approximation sets are nonempty, defining, in addition, a polygonal curve of finite length, see \cite{DDDL2015} \textit{e.g.}). The property of self-contractedness is completely lost in case of inexact projections or in case of variational perturbations of the sets, changing the nature of the problem. Therefore, an interesting question is to determine whether there exists a weaker regularity condition on $(A,B)$ that is equivalent to the norm convergence of the alternating projection method. \smallskip\newline
	Another possible direction for further research concerns regularity and $d$-stability for more than two closed convex sets or for different iterative projection algorithms (see \cites{BCC2012, CRW2019} \textit{e.g.}). Indeed, the convex feasibility problem for more than two sets is a rich field of research, where the order (and the frequency) on which the alternating projections are taken would potentially lead to different conclusions and unexplored notions of regularity.
	\smallskip\newline

	\textbf{Acknowledgement.} Major part of this research has been done during a research visit of the second author to the Università Cattolica del Sacro Cuore (Milano, June 2025). This author thanks his hosts for hospitality. The second author thanks Stoyan Apostolov and Eva Kopeck\'a for useful discussions. The authors thank anonymous referees whose comments and remarks considerably improved the presentation of this work. \smallskip\newline
    The first author was partially supported by INdAM-GNAMPA. The research of the second author was supported by the INdAM (GNAMPA, Professore Visitatore) and the Austrian Science Fund (Grant FWF 10.55776/P36344). The research of the third and fourth authors was supported by the INdAM -GNAMPA Project, CUP E53C23001670001, and by the MICINN project PID2020-112491GB-I00 (Spain). \smallskip\newline
	For open access purposes, the second author has applied a CC BY public copyright license to any author-accepted manuscript version arising from this submission.
	\bigskip
	
	\textbf{Declarations} \smallskip\newline
	\textit{Conflict of interests/Competing interests.} The authors declare that they have no relevant financial or non-financial conflicts of interest to disclose.
	

\end{document}